\documentclass[10pt]{article}

\usepackage{amsthm,amsfonts,amsmath,amssymb,fullpage}

\nonstopmode \numberwithin{equation}{section}
\newtheorem{definition}{Definition}[section]
\newtheorem{theorem}{Theorem}

\newtheorem{lemma}{Lemma}[section]

\newtheorem{remark}{Remark}[section]

\allowdisplaybreaks

\begin{document}
\title{\textbf{Inequalities of extended $(p,q)$-beta and confluent hypergeometric function}}
\author{Shahid Mubeen$^{1}$, Kottakkaran Sooppy Nisar$^{2,*}$, Gauhar Rahman$^{3}$ and\\ Muhammad Arshad$^{3}$ }
\date{}
\maketitle

\noindent $^{1}$ Department of Mathematics, University of Sargodha,   Sargodha, Pakistan, email: smjhanda@gmail.com\newline
\noindent $^{2}$ Department of Mathematics, College of Arts \& Science-Wadi Addawaser, Prince Sattam bin Abdulaziz University, Saudi Arabia, email: n.sooppy@psau.edu.sa\newline
\noindent $^{3}$ Department of Mathematics, International Islamic  University, Islamabad, Pakistan,\\ 
emails: gauhar55uom@gmail.com (G. Rahman), marshad\_zia@yahoo.com (M. Arshad)

\begin{abstract}
 In this present paper, we establish the log convexity  and Tur\'{a}n type inequalities of extended $(p,q)$-beta functions. Also, we present the log-convexity, the monotonicity and Tur\'{a}n type inequalities for extended $(p,q)$-confluent hypergeometric function by using the inequalities of extended $(p,q)$-beta functions.
\end{abstract}

\textbf {keywords}: Extended beta functions, extended hypergeometric functions, log-convexity, Tur\'{a}n-type inequalities.


\section{Introduction}\label{Intro1}
We begin with the classical gamma function
 \begin{eqnarray*}
 \Gamma(z)=\int\limits_{0}^{\infty}t^{z-1}e^{-t}dt,\quad \Re(z)>0.
 \end{eqnarray*}
 In another way, it is defined as
 \begin{eqnarray*}
\Gamma(z)=\lim_{n \rightarrow \infty}\frac{n!n^{z-1}}{(z)_{n}}
\end{eqnarray*}
where $(\alpha)_n$ is the Pochhammer symbol defined as
\begin{eqnarray*}
 (\alpha)_{n}= \left\{\begin{array}{l}\label{1}
 \alpha(\alpha+1)(\alpha+2)\cdots(\alpha+n-1); \quad for \quad
  n\geq1, \alpha\neq0\\
  1 \quad if \quad n=0
  \end{array}\right.
  \end{eqnarray*}
 and
\begin{eqnarray*}
\Gamma(z+1)=z\Gamma(z)
\end{eqnarray*}
The relation between Pochhammer symbol and gamma function is given below
\begin{eqnarray*}
(z)_{n}=\frac{\Gamma(z+n)}{\Gamma(z)}.
\end{eqnarray*}
The beta function is defined by
\begin{eqnarray}\label{beta}
B(x,y)=\int\limits_{0}^{1}t^{x-1}(1-t)^{y-1}dt,
\end{eqnarray}
$(\Re(x)>0, \Re(y)>0)$\\
and
\begin{eqnarray}
B(x,y)=\frac{\Gamma(x)\Gamma(y)}{\Gamma(x+y)}, \Re(x)>0, \Re(y)>0.
\end{eqnarray}
Chaudhry and Zubair \cite{CH} and Chaudhry et al. \cite{CHS} defined the following extended gamma and beta functions
\begin{eqnarray}\label{Egamma}
 \Gamma_p(z)=\int\limits_{0}^{\infty}t^{z-1}e^{-t-pt^{-1}}dt,
 \end{eqnarray}
 $\Re(z)>0, p\geq0$. When $p=0$, then $\Gamma_p$  tends to the classical gamma function $\Gamma$,\\
and
\begin{eqnarray}\label{Ebeta}
B(x,y;p)=\int\limits_{0}^{1}t^{x-1}(1-t)^{y-1}e^{-\frac{p}{t(1-t)}}dt
\end{eqnarray}
(where $\Re(p)>0, \Re(x)>0, \Re(y)>0$) respectively. When $p=0$, then $B(x,y;0)=B(x,y)$.
Recently Choi et al. \cite{Choi2014} introduced the following extension of extended beta function as
\begin{eqnarray}\label{EEbeta}
B(x,y;p,q)=B_{p,q}(x,y)=\int\limits_{0}^{1}t^{x-1}(1-t)^{y-1}e^{-\frac{p}{t}-\frac{q}{1-t}}dt
\end{eqnarray}
(where $\Re(p)>0, \Re(q)>0, \Re(x)>0, \Re(y)>0$).

It is clear that when $p=q$, then (\ref{EEbeta}) reduces to the well known extended beta function (\ref{Ebeta}). Similarly if $p=q=0$, then  (\ref{EEbeta}) reduces to the classical beta function (\ref{beta}).
In the same paper, they also defined the following extension of extended  confluent hypergeometric function by
\begin{eqnarray}\label{EChyp}
\Phi_{p,q}\Big(\beta;\gamma;z\Big)=\sum\limits_{n=0}^{\infty}\frac{B(\beta+n;\gamma-\beta;p,q)}{B(\beta,\gamma-\beta)}
\frac{z^n}{n!}
\end{eqnarray}
$$\Big( p\geq0, q\geq0,  \Re(\gamma)>\Re(\beta)>0 \Big).$$
The integral representations of extension of extended  confluent hypergeometric function is given by
\begin{align}\label{intch}
\Phi_{p,q}\Big(\beta;\gamma;z\Big)=\frac{1}{B(\beta,\gamma-\beta)}\int_0^1t^{\beta-1}(1-t)^{\gamma-\beta-1}
\exp\Big(zt-\frac{p}{t}-\frac{q}{1-t}\Big)dt
\end{align}
$$\big(p\geq0, q\geq0,  \Re(\gamma)>\Re(\beta)>0\Big).$$

Note that for $p=q$,  the series  (\ref{EChyp}) respectively reduces to the extended confluent hypergeometric series. Similarly for $p=q=0$ the series  (\ref{EChyp}) respectively reduces to the classical confluent hypergeometric series.


\section{Main results: Inequalities of extended $(p,q)$-beta function}\par
In this section, we establish some inequalities which involve  extended $(p,q)$-beta functions
by using some natural inequalities \cite{CB}. For this continuation of our study, we recall the following well-known Chebychev's integral inequality and H\"{o}lder-Rogers inequality.
 \begin{lemma}(see \cite{Dragomir,AG})
 Let the functions $f,g:[a,b]\subseteq {R}\rightarrow {R}$ are asynchronous for all $x\in [a,b]$  and $p(x):[a,b]\subseteq {R}\rightarrow {R}$ is a positive integrable function, then
\begin{eqnarray}\label{Cheby}
\int\limits_{a}^{b}p(x)f(x)dx\int\limits_{a}^{b}p(x)g(x)dx\leq\int\limits_{a}^{b}p(x)dx\int\limits_{a}^{b}p(x)f(x)g(x)dx.
\end{eqnarray}
\end{lemma}
\begin{definition} In \cite{AT}, a function $f:(a,b)\rightarrow {R}$ is said to be convex if for any $x_1,x_2\in (a,b)$ and $\alpha\in(0,1)$
\begin{eqnarray}
f(\alpha x_1+(1-\alpha)x_2)\leq\alpha f(x_1)+(1-\alpha)f(x_2).
\end{eqnarray}
It shows that when we move from $x_1$ to $x_2$, the line joining the points $(x_1,f(x_1))$ and $(x_2,f(x_2))$ lies always above the graph of $f$.
\end{definition}
\begin{definition} A function $f$ is said to be a log-convex if $f>0$ and $\log f$ is convex \emph{i.e.},
for all $x_1,x_2\in I$ (where I is an interval) and $\alpha\in (0,1)$, we have
\begin{eqnarray*}
\log f(\alpha x_1 +(1-\alpha)x_2)\leq\alpha \log f(x_1)+(1-\alpha)\log f(x_2)=\log(f^{\alpha}(x)f^{1-\alpha}(x_2)).
\end{eqnarray*}
This implies that
\begin{eqnarray}
f(\alpha x_1 +(1-\alpha)x_2)\leq f^{\alpha}(x_1)f^{1-\alpha}(x_2).
\end{eqnarray}
\end{definition}
\begin{lemma}\label{lem2} (H\"{o}lder inequality  \cite{LJ5}) If $\theta_1$ and $\theta_2$ are positive real numbers such that $\frac{1}{\theta_1}+\frac{1}{\theta_2}=1$, then the following inequality holds for integrable functions $f,g:[a,b]\rightarrow {R}$:
\begin{eqnarray}\label{Holder}
|\int\limits_{a}^{b}f(x)g(x)dx|\leq(\int\limits_{a}^{b}|f|^{\theta_1}dx)^{\frac{1}{\theta_1}}(\int\limits_{a}^{b}|g|^{\theta_2}dx)^{\frac{1}{\theta_2}}.
\end{eqnarray}
\end{lemma}
\begin{theorem}\label{th1} If $x,y, x_1,y_1$  are positive real numbers satisfying the condition
\begin{eqnarray}
(x-x_1)(y-y_1)\geq0,
\end{eqnarray}
then for the extended $(p,q)$-beta function, we have the inequality
\begin{eqnarray}\label{eq1}
B_{p,q}(x,y_1)B_{p,q}(x_1,y)\leq B_{p,q}(x_1,y_1)B_{p,q}(x,y),
\end{eqnarray}
\end{theorem}
\begin{proof} Consider the mappings $f,g,h:[0,1]\rightarrow[0,\infty)$ given by\\
$f(t)=t^{x-x_1}$, $g(t)=(1-t)^{y-y_1}$ and $h(t)= t^{x_1-1}(1-t)^{y_1-1}\exp\Big(-\frac{p}{t}-\frac{q}{1-t}\Big)$.\\
Now, differentiation of $f$ and $g$ gives
\begin{eqnarray*}
f^{\prime}(t)=(x-x_1)t^{x-x_1-1},\quad g^{\prime}(t)=(y_1-y)(1-x)^{y-y_1-1}.
\end{eqnarray*}
This show that $f$ and $g$ have the same monotonicity on $[0,1]$. \\
Applying the Chebyshev's integral inequality (\ref{Cheby}), for the above defined functions $f$, $g$ and $h$, we have
\begin{align*}
&\Big(\int_a^bt^{x-1}(1-t)^{y_1-1}\exp\Big(-\frac{p}{t}-\frac{q}{1-t}\Big)dt\Big)
\Big(\int_a^bt^{x_1-1}(1-t)^{y-1}\exp\Big(-\frac{p}{t}-\frac{q}{1-t}\Big)dt\Big)\\
&\leq\Big(\int_a^bt^{x_1-1}(1-t)^{y_1-1}\exp\Big(-\frac{p}{t}-\frac{q}{1-t}\Big)dt\Big)
\Big(\int_a^bt^{x-1}(1-t)^{y-1}\exp\Big(-\frac{p}{t}-\frac{q}{1-t}\Big)dt\Big)
\end{align*}
which implies that,
\begin{eqnarray*}
B_{p,q}(x,y_1)B_{p,q}(x_1,y)&\leq& B_{p,q}(x_1,y_1)B_{p,q}(x,y)
\end{eqnarray*}
which completes the desired proof.
\end{proof}
\begin{theorem}\label{th2}
The function $(p,q)\mapsto B_{p,q}(x,y)$ is log convex on $(0,\infty)$ for each $x,y>0$.
Moreover, the function $B_{p,q}(x,y)$ satisfy the following Tur\'{a}n type inequality
\begin{eqnarray}
B_{p,q}^2(x,y)-B_{p+a,q+a}(x,y)B_{p-a,q-a}(x,y)\leq0,
\end{eqnarray}
for all real $a$.
\end{theorem}
\begin{proof}
From the definition of log-convexity, it will be sufficient to prove that
\begin{eqnarray}\label{eq2}
 B_{\alpha p_1+(1-\alpha) p_2,\alpha q_1+(1-\alpha)q_2}(x,y)\leq\Big( B_{p,q}(x,y)\Big)^\alpha\Big( B_{p,q}(x,y)\Big)^{1-\alpha},
 \end{eqnarray}
 for $\alpha\in[0,1]$, $p_1,p_2, q_1, q_2>0$ and for a fixed $x,y>0$.
 Obviously, (\ref{eq2}) is true for $\alpha=0$ and $\alpha=1$. Assume that $\alpha\in(0,1)$, then it follows from (\ref{EEbeta}) that
 \begin{align}
&B_{\alpha p_1+(1-\alpha) p_2,\alpha q_1+(1-\alpha)q_2}(x,y)\notag\\
&=\int_0^1t^{x-1}(1-t)^{y-1}\exp\Big(\frac{-\alpha p_1-(1-\alpha)p_2}{t}+\frac{-\alpha q_1-(1-\alpha)q_2}{1-t}\Big)\notag\\
&=\Big(\int_0^1t^{x-1}(1-t)^{y-1}\exp\Big(-\frac{ p_1}{t}-\frac{q_1}{1-t}\Big)\Big)^\alpha\notag\\
&\times\Big(\int_0^1t^{x-1}(1-t)^{y-1}\exp\Big(-\frac{ p_1}{t}-\frac{q_1}{1-t}\Big)\Big)^{1-\alpha} \label{eq3}
 \end{align}
Let $\theta_1=\frac{1}{\alpha}$ and $\theta_2=\frac{1}{(1-\alpha)}$. Clearly $\theta_1>1$ and $\theta_1+\theta_2=\theta_1\theta_2$. Thus applying the H\"{o}lder-Rogers inequality (\ref{Holder})  for integrals in (\ref{eq3}) gives
\begin{align}
B_{\alpha p_1+(1-\alpha) p_2,\alpha q_1+(1-\alpha)q_2}(x,y)&<
\Big(\int_0^1t^{x-1}(1-t)^{y-1}\exp\Big(-\frac{ p_1}{t}-\frac{q_1}{1-t}\Big)\Big)^\alpha\notag\\
&\times\Big(\int_0^1t^{x-1}(1-t)^{y-1}\exp\Big(-\frac{ p_1}{t}-\frac{q_1}{1-t}\Big)\Big)^{1-\alpha}\notag\\
&=\Big( B_{p,q}(x,y)\Big)^\alpha\Big( B_{p,q}(x,y)\Big)^{1-\alpha}, \label{eq4}
\end{align}
This implies that $(p,q)\mapsto B_{p,q}(x,y)$ is log convex on $(0,\infty)$.\\
Now, taking $\alpha=\frac{1}{2}$, $p_1=p-a$, $p_2=p+a$, and $q_1=q-a$, $q_2=q+a$, the inequality (\ref{eq4}) yields
\begin{eqnarray*}
B_{p,q}^2(x,y)-B_{p+a,q+a}(x,y)B_{p-a,q-a}(x,y)\leq0.
\end{eqnarray*}
\end{proof}
\begin{theorem}\label{th3}
The function $(x,y)\mapsto B_{p,q}(x,y)$ is logarithmic convex on $(0,\infty)\times(0,\infty)$, for all $p,q\geq0$. In particular
$$B_{p,q}^2\Big(\frac{x_1+x_2}{2}, \frac{y_1+y_2}{2}\Big)\leq B_{p,q}(x_1,y_1)B_{p,q}(x_2,y_2).$$
\end{theorem}
\begin{proof}
Let $(x_1,y_1),(x_2,y_2)\in (0,\infty)^2$, and $c,d\geq0$ with $c+d=1$, then we have
\begin{eqnarray}\label{eq5}
B_{p,q}\Big(c(x_1,y_1)+d(x_2,y_2)\Big)=B_{p,q}(cx_1+dx_2, cy_1+dy_2).
\end{eqnarray}
Applying the definition of $(p,q)$-extended beta function on the right hand side of inequality (\ref{eq5}), we have
\begin{align*}
&B_{p,q}\Big(c(x_1,y_1)+d(x_2,y_2)\Big)\\
&=\int\limits_{0}^{1}t^{cx_1+dx_2-1}(1-t)^{cy_1+dy_2-1}\exp\Big(-\frac{p}{t}-\frac{q}{1-t}\Big)dt\\
&=\int\limits_{0}^{1}t^{cx_1+dx_2 -(c+d)}(1-t)^{cy_1+dy_2-(c+d)}\exp\Big(-\frac{p(c+d)}{t}-\frac{q(c+d)}{1-t}\Big)dt\\
&=\int\limits_{0}^{1}t^{c(x_1-1)}t^{d(x_2-1}(1-t)^{c(y_1-1)}(1-t)^{d(y_2-1)}\exp\Big(-\frac{pc}{t}-\frac{qc}{1-t}\Big)\exp\Big(-\frac{pd}{t}-\frac{qd}{1-t}\Big)dt\\
&=\int\limits_{0}^{1}\Big(t^{x_1-1}(1-t)^{y_1-1}\exp\Big(-\frac{p}{t}-\frac{q}{1-t}\Big)\Big)^{c}\Big(t^{x_2-1}t^{y_2-1}\exp\Big(-\frac{p}{t}-\frac{q}{1-t}\Big)\Big)^{d}dt.
\end{align*}
Again by considering $\theta_1=\frac{1}{c}$, $\theta_2=\frac{1}{d}$, we can use the H\"{o}lder-Rogers inequality for above integrals and it follows
\begin{eqnarray*}
B_{p,q}\Big(c(x_1,y_1)+d(x_2,y_2)\Big)&\leq&
\Big(\int\limits_{0}^{1}t^{x_1-1}(1-t)^{y_1-1}\exp\Big(-\frac{p}{t}-\frac{q}{1-t}\Big)dt\Big)^{c}
\\&\times&\Big(\int\limits_{0}^{1}t^{x_2-1}t^{y_2-1}\exp\Big(-\frac{p}{t}-\frac{q}{1-t}\Big)dt\Big)^{d}\\
&=&\Big(B_{p,q}(x_1,y_1)\Big)^c\Big(B_{p,q}(x_2,y_2)\Big)^d.
\end{eqnarray*}
 This shows the logarithmic convexity of extended $(p,q)$-beta function $B_{p,q}(x,y)$ on $(0,\infty)^2$.\\
For $c=d=\frac{1}{2}$, the above inequality reduces to
\begin{eqnarray}\label{eq6}
B_{p,q}^2\Big( \frac{x_1+x_2}{2}, \frac{y_1+y_2}{2}\Big)\leq B_{p,q}(x_1,y_1)B_{p,q}(x_2,y_2).
\end{eqnarray}
Let $x,y>0$ be such that $\min_{a\in{R}}(x+a, x-a)>0$, then by taking $x_1=x+a$, $x_2=x+a$, $y_1=y+b$ and
$y_2=y-b$ in (\ref{eq6}), we get
\begin{eqnarray}
\Big[B_{p,q}(x, y)\Big]^2\leq B_{p,q}(x+a, y+b)B_{p,q}(x-a, y-b),
\end{eqnarray}
for all $p,q\geq0$.
\end{proof}
\section{Inequalities for $(p,q)$-extended confluent hypergeometric function}
In this section, we present the log-convexity and Tur\'{a}n type inequality for extended confluent hypergeometric function defined in (\ref{EChyp}). For this continuation, we recall the following well-known lemma.
\begin{lemma}\label{lem3}\cite{Biernacki} Consider the power series $f(x)=\sum_{n\geq0}a_nx^n$ and $g(x)=\sum_{n\geq0}b_nx^n$, where $a_n\in {R}$ and $b_n>0$ for all $n$. Further assume that both series converge on $|x|<\alpha$. If the sequence $\{a_n/b_n\}_n\geq0$ is increasing (or decreasing), then $x\mapsto f(x)/g(x)$ is also increasing (or decreasing) function on $(0,\alpha)$.
\end{lemma}
Note  that the above lemma is valid only if both $f$ and $g$ are both even or both  odd functions.
\begin{theorem}
Let $\beta\geq0$ and $\gamma,\delta>0$, then the following assertions for extended $(p,q)$-confluent hypergeometric function are true.\\
(i) For $\gamma\geq \delta$, the function $x\mapsto\Phi_{p,q}\Big(\beta;\gamma;x\Big)/\Phi_{p,q}\Big(\beta;\delta;x\Big)$ is increasing on $(0,\infty)$.\\
(ii)  For $\gamma\geq \delta$, we have \\$\delta\Phi_{p,q}\Big(\beta+1;\gamma+1;x\Big)\Phi_{p,q}\Big(\beta;\delta;x\Big)\geq \gamma\Phi_{p,q}(\beta;\gamma;x)\Phi_{p,q}
\Big(\beta+1;\delta+1;x\Big)$.\\
(iii) The function $x\mapsto\Phi_{p,q}\Big(\beta;\gamma;x\Big)$ is log-convex on ${R}$.\\
(iv) The function $(p,q)\mapsto \Phi_{p,q}\Big(\beta;\gamma;x\Big)$ is log convex on $(0,\infty)$ for fixed $x>0$.\\
(v) Let $\sigma>0$. then the function $$\beta\mapsto\frac{B(\beta,\gamma)\Phi_{p,q}\Big(\beta+\sigma;\gamma;x\Big)}{B(\beta+\sigma,\gamma)\Phi_{p,q}\Big(\beta;\gamma;x\Big)}$$
is decreasing on $(0,\infty)$ for fixed $\gamma,x>0$.
\end{theorem}
\begin{proof}
From the definition of (\ref{EChyp}), it follows that
\begin{align}
\frac{\Phi_{p,q}\Big(\beta;\gamma;x\Big)}{\Phi_{p,q}\Big(\beta;\delta;x\Big)}=\frac{\sum_{n=0}^{\infty}a_n(c)x^n}
{\sum_{n=0}^{\infty}a_n(d)x^n}, where \quad a_n(t)=\frac{B_{p,q}(\beta+n,t-\beta)}{B(\beta,t-\beta)n!}.
\end{align}
If we denote $f_n=a_n(c)/a_n(d)$, then
\begin{eqnarray*}
f_n-f_{n+1}&=&\frac{a_n(c)}{a_n(d)}-\frac{a_{n+1}(c)}{a_{n+1}(d)}\\
&=&\frac{B(\beta,\delta-\beta)}{B(\beta,\gamma-\beta)}\Big(\frac{B_{p,q}(\beta+n,\gamma-\beta)}{B_{p,q}(\beta+n,\delta-\beta)}
-\frac{B_{p,q}(\beta+n+1,\gamma-\beta)}
{B_{p,q}(\beta+n+1,\delta-\beta)}\Big).
\end{eqnarray*}
Now take $x=\beta+n$, $y=\delta-\beta$, $x_1=\beta+n+1$, $y_1=\gamma-\beta$ in (\ref{eq1}). Since $(x-x_1)(y-y_1)=\gamma-\delta\geq0$, it follows from Theorem \ref{th1} that
\begin{eqnarray*}
\frac{B_{p,q}(\beta+n,\gamma-\beta)}{B_{p,q}(\beta+n,\delta-\beta)}\leq\frac{B_{p,q}(\beta+n+1,\gamma-\beta)}
{B_{p,q}(\beta+n+1,\delta-\beta)},
\end{eqnarray*}
 this is equivalent to say that  $\{f_n\}$ is an increasing  sequence and hence with the aid of Lemma \ref{lem3}, we observe that $x\mapsto\Phi_{p,q}\Big(\beta;\gamma;x\Big)/\Phi_{p,q}\Big(\beta;\delta;x\Big)$ is increasing on $(0,\infty)$.\\
To prove the assertion (ii), we recall the following well-known identity from \cite{Choi2014}:
\begin{eqnarray}\label{eq7}
\frac{d^n}{dx^n}\Phi_{p,q}\Big(\beta;\gamma;x\Big)=\frac{(\beta)_n}{(\gamma)_n}\Phi_{p,q}\Big(\beta+n;\gamma+n;x\Big).
\end{eqnarray}
Since the increasing properties of $x\mapsto\Phi_{p,q}\Big(\beta;\gamma;x\Big)/\Phi_{p,q}\Big(\beta,\delta;x\Big)$ is
equivalent to the following inequality
\begin{eqnarray}\label{eq8}
\frac{d}{dx}\Big(\frac{\Phi_{p,q}\Big(\beta;\gamma;x\Big)}{\Phi_{p,q}\Big(\beta,\delta;x\Big)}\Big)\geq0.
\end{eqnarray}
This together with (\ref{eq7}) implies
\begin{eqnarray*}
\Phi_{p,q}^\prime\Big(\beta;\gamma;x\Big)\Phi_{p,q}\Big(\beta;\delta;x\Big)&-&\Phi_{p,q}\Big(\beta;\gamma;x\Big)
\Phi_{p,q}^\prime\Big(\beta;\delta;x\Big)\\
&=&\frac{\beta}{\gamma}\Phi_{p,q}\Big(\beta+1;\gamma+1;x\Big)\Phi_{p,q}(\beta;\delta;x)\\&-&\frac{\beta}{\delta}
\Phi_{p,q}\Big(\beta;\gamma;x\Big)\Phi_{p,q}
\Big(\beta+1;\delta+1;x\Big)\geq0.
\end{eqnarray*}
This implies that
\begin{eqnarray*}
\delta\Phi_{p,q}\Big(\beta+1;\gamma+1;x\Big)\Phi_{p,q}(\beta;\delta;x)\geq \gamma\Phi_{p,q}\Big(\beta;\gamma;x\Big)\Phi_{p,q}
\Big(\beta+1;\delta+1;x\Big)
\end{eqnarray*}
which prove the assertion.
The log-convexity of $x\mapsto \Phi_{p,q}\Big(\beta;\gamma;x\Big)$ can be prove by using the integral representation of extended $(p,q)$-confluent hypergeometric function as given in (\ref{intch}) and by applying the H\"{o}lder-Rogers inequality for integrals as follows:
\begin{align*}
&\Phi_{p,q}\Big(\beta;\gamma;\alpha x+(1-\alpha)y\Big)\\
&=\frac{1}{B(\beta,\gamma-\beta)}\int_0^1t^{\beta-1}(1-t)^{\gamma-\beta-1}\exp\Big(\alpha xt+(1-\alpha)yt-\frac{p}{t}-\frac{q}{1-t}\Big)dt\\
&=\frac{1}{B(\beta,\gamma-\beta)}\int_0^1\Big[\Big(t^{\beta-1}(1-t)^{\gamma-\beta-1}\exp\Big( xt-\frac{p}{t}-\frac{q}{1-t}\Big)\Big)^\alpha\\
&\times\Big(t^{\beta-1}(1-t)^{\gamma-\beta-1}\exp\Big( yt-\frac{p}{t}-\frac{q}{1-t}\Big)\Big)^{1-\alpha}\Big]dt\\
&\leq \Big[\frac{1}{B(\beta,\gamma-\beta)}\int_0^1t^{\beta-1}(1-t)^{\gamma-\beta-1}\exp\Big( xt-\frac{p}{t}-\frac{q}{1-t}\Big) dt\Big]^\alpha\\
&\times\Big[\frac{1}{B(\beta,\gamma-\beta)}\int_0^1t^{\beta-1}(1-t)^{
\gamma-\beta-1}\exp\Big( xt-\frac{p}{t}-\frac{q}{1-t}\Big)dt\Big]^{1-\alpha}\\
&=\Big(\Phi_{p,q}\Big(\beta;\gamma;x\Big)\Big)^\alpha\Big(\Phi_{p,q}\Big(\beta;\gamma;y\Big)\Big)^{1-\alpha}, (x,y>0, \alpha\in[0,1]).
\end{align*}
 This prove that $x\mapsto\Phi_{p,q}\Big(\beta;\gamma;x\Big)$ is log-convex for a fixed $x>0$. For the case when $x<0$, then the assertion immediately follows from the identity (see \cite{Choi2014}):
\begin{eqnarray*}
\Phi_{p,q}\Big(\beta;\gamma;x\Big)= e^x\Phi_{q,p}\Big(\gamma-\beta;\gamma;-z\Big).
\end{eqnarray*}
Since, the infinite sum of log-convex functions is log-convex for $x>0$. Thus, the log-convexity of $(p,q)\mapsto\Phi_{p,q}\Big(\beta;\gamma;x\Big)$ is equivalent to prove that  $(p,q)\mapsto B(\beta+n,\gamma-\beta)$ is log-convex on $(0,\infty)$ and for non-negative integer $n$. From Theorem \ref{th2}, it is clear that  $(p,q)\mapsto B(\beta+n,\gamma-\beta)$ is log-convex for $\gamma>\beta>0$ and hence assertion (iv) is true.

Now, let $\beta^\prime\geq \beta$ and set $h(t)=t^{\beta^\prime-1}(1-t)^{\gamma-\beta^\prime-1}\exp\Big(xt-\frac{p}{t}-\frac{q}{1-t}\Big)$,
$f(t)=\Big(\frac{t}{1-t}\Big)^{\beta-\beta^\prime}$ and $g(t)=\Big(\frac{t}{1-t}\Big)^{\sigma}$. Then using the integral representation (\ref{intch}) of extended confluent hypergeometric function, we have
\begin{align}\label{eq9}
&\frac{B(\beta,\gamma)\Phi_{p,q}\Big(\beta+\sigma;\gamma;x\Big)}{B(\beta+\sigma,c)\Phi_{p,q}\Big(\beta;\gamma;x\Big)}
-\frac{B(\beta^\prime,\gamma)\Phi_{p,q}\Big(\beta^\prime+\sigma;\gamma;x\Big)}{B(\beta^\prime+\sigma,\gamma)
\Phi_{p,q}\Big(\beta^\prime;\gamma;x\Big)}\notag\\
&=\frac{\int_0^1f(t)g(t)h(t)dt}{\int_0^1f(t)h(t)dt}-\frac{\int_0^1g(t)h(t)dt}{\int_0^1h(t)dt}.
\end{align}
 One can easily determine that for $\beta^\prime\geq \beta$, the function $f$ is decreasing when $\sigma\geq0$ and the function  $g$ is increasing . Since $h$ is non negative function for $t\in[0,1]$. Thus, by reverse Chebyshev's reverse inequality (\ref{Cheby}), it follows that
\begin{eqnarray}
\int_0^1f(t)h(t)dt\int_0^1g(t)h(t)dt\leq \int_0^1h(t)dt\int_0^1f(t)g(t)h(t)dt.
\end{eqnarray}
This together with (\ref{eq9}) implies
\begin{eqnarray*}
\frac{B(\beta,\gamma)\Phi_{p,q}\Big(\beta+\sigma;\gamma;x\Big)}{B(\beta+\delta,\gamma)
\Phi_{p,q}\Big(\beta;\gamma;x\Big)}
-\frac{B(\beta^\prime,\gamma)\Phi_{p,q}\Big(\beta^\prime+\sigma;\gamma;x\Big)}{B(\beta^\prime+\sigma,\gamma)
\Phi_{p,q}
\Big(\beta^\prime;\gamma;x\Big)}\geq0,
\end{eqnarray*}
which is equivalent to say that the function
\begin{eqnarray*}
\beta\mapsto\frac{B(\beta,\gamma)\Phi_{p,q}\Big(\beta+\sigma;\gamma;x\Big)}{B(\beta+\sigma,\gamma)
\Phi_{p,q}\Big(\beta;\gamma;x\Big)}
\end{eqnarray*}
is decreasing on $(0,\infty)$.
\end{proof}
\begin{remark}
In particular, the following decreasing property of extended $(p,q)$-confluent hypergeometric function
\begin{eqnarray*}
\beta\mapsto\frac{B(\beta,\gamma)\Phi_{p,q}\Big(\beta+\sigma;\gamma;x\Big)}{B(\beta+\sigma,\gamma)
\Phi_{p,q}\Big(\beta;\gamma;x\Big)}
\end{eqnarray*}
is equivalent to the following inequality
\begin{align}
\Phi_{p,q}^2\Big(\beta+\sigma;\gamma;x\Big)\geq \frac{B^2(\beta+\sigma,\gamma)}{B(\beta+2\sigma,\gamma)B(\beta,\gamma)}\Phi_{p,q}\Big(\beta+2\sigma;\gamma;x\Big)
\Phi_{p,q}\Big(\beta;\gamma;x\Big).
\end{align}
When $p=q$, then the above inequality will reduce to the inequality recently proved by \cite{Mondal}. Similarly, when $p=q=0$, then the above inequality reduces to the inequality of confluent hypergeometric which is an improved version of Theorem 4(b) given in \cite{Karp}.
\end{remark}

\section{conclusion}
In this paper, we introduced inequalities for extended $(p,q)$-beta and $(p,q)$-confluent hypergeometric function defined by Choi et al. \cite{Choi2014}. Throughout in this paper, if we take $p=q$ then we get the inequalities of extended beta
function and extended confluent hypergeometric function recently introduced by Mondal \cite{Mondal}. Similarly if we take $p=q=0$, then the newly defined inequalities for extended $(p,q)$-beta function will reduce to the  inequalities of classical beta function (see \cite{Agarwal,Dragomir}).







%

\end{document}